\documentclass{amsart}
\usepackage{enumitem}
\usepackage{latexsym}
\usepackage{amsmath,amsfonts,amsthm, amssymb}
\input{xy}
\xyoption{all}

\newtheorem{theorem}{Theorem}[section]
\newtheorem{lemma}[theorem]{Lemma}

\newtheorem*{varthmdfold}{Theorem {\ref{dfold}}}
\newtheorem*{vtsmooththm}{Theorem {\ref{smooththm}}}

\newtheorem*{condI}{Condition $1\beta$}
\newtheorem*{condII}{Condition 2}

\newtheorem*{condi}{Condition i}
\newtheorem*{condii}{Condition ii}
\begin{document}

\title{A Shiu Theorem for Larger and Smoother Functions}
\author[T. Wright]{Thomas Wright}
\begin{abstract}
In this paper, we broaden Shiu's Brun-Titchmarsh theorem to allow for functions that are larger and/or smooth-supported.  In particular, let $x$ be a large parameter, and let $f_x(n)$ be a nonnegative multiplicative function parametrized by $x$.  We prove that if there exists a small constant $\beta$ and a constant $A_1$ such that $f_x(p^l)\leq (A_1\log\log x)^{l\beta}$ for every prime $p\leq x$ and every $l\geq 1$, and if $f_x(n)\ll \max\{n^\epsilon,(\log x)^\epsilon\}$ for every $\epsilon>0$ and every $n\leq x$, then
$$\sum_{\substack{x-y\leq n\leq x \\ n\equiv a\pmod k}}f_x(n)\ll_{\epsilon_0} \frac{y}{\phi(k)(\log x)^{1-\epsilon_0}}\exp\left(\sum_{\substack{p\leq x \\ p\nmid k}}\frac{f(p)}{p}\right)$$
for every $\epsilon_0>0$, where $x^\kappa\leq y\leq x$, $k<y^{1-\alpha}$, and $(a,k)=1$ for relatively small $\alpha$ and $\kappa$, and the implied constant does not depend upon the parameter $x$.  Moreover, we prove that if $f$ is a $Q$-smooth-supported function then there exists a constant $C$ for which
$$\sum_{\substack{x-y\leq n\leq x \\ n\equiv a\pmod k}}f_x(n)\ll_{\epsilon_0} \frac{y}{\phi(k)(\log x)^{1-\epsilon_0}}\exp\left(\sum_{\substack{p\leq x \\ p\nmid k}}\frac{f(p)}{p}\right)\rho(u)^C,$$
where $u=\frac{\log x}{\log Q}$, $\rho$ is the Dickman-de Bruijn function, and $C$ depends on whether $f_x(p)$ is bounded by a constant or by a slowly-growing function of $x$.

We also give applications to both the divisor function to large powers and to smooth numbers in short intervals.
\end{abstract}
\maketitle

\section{Introduction}

One of the most widely-used bounds in the theory of multiplicative functions was proven by Peter Shiu in 1980 \cite{Sh}, who proved a Brun-Titchmarsh-type result for these functions.  We begin by stating this result below.

Let $\mathcal M$ denote the set of nonnegative multiplicative functions $f$ such that the following two conditions hold.  

\begin{condi}
There exists a constant $A_1$ such that for any prime $p$ and any natural number $l$,
$$f(p^l)\leq A_1^l.$$
\end{condi}
\begin{condii}
For any $\epsilon>0$ there exists a constant $A_2=A_2(\epsilon)$ such that for any $n$,
$$f(n)\leq A_2n^\epsilon.$$
\end{condii}

Shiu's theorem is then the following.

\begin{theorem}[Shiu, 1980]\label{shiu}
Let $f\in \mathcal M$.  Let $0<\alpha<\frac 12$ and $0<\kappa<\frac 12$, and let $x^\kappa\leq y\leq x$ and $k<y^{1-\alpha}$.  Moreover, let $a$ be such that $(a,k)=1$.  Then for all sufficiently large $x$,
$$\sum_{\substack{x-y\leq n\leq x \\ n\equiv a\pmod k}}f(n)\ll \frac{y}{\phi(k)\log x}\exp\left(\sum_{\substack{p\leq x\\ p\nmid k}}\frac{f(p)}{p}\right).$$
Here, the term "sufficiently large" depends only on the choices of $\alpha$ and $\kappa$ and the constants $A_1$ and $A_2(\epsilon)$ for some reasonable choice of $\epsilon$ that depends on $\alpha$ and $\kappa$.
\end{theorem}

However, there are a few limitations to this theorem that we wish to address in the current paper.  First, this theorem cannot say anything when $f(p)$ is unbounded.  Second, the result is limited in the case where $f$ is a smooth-supported function, since smooth numbers are generally far more sparse than this bound would indicate.  In this paper, we expand the scope of the theorem to handle both such cases.

\section{Introduction: Larger Functions}\label{larger}

Let us begin with the idea of letting $f(p)$ go to infinity.  Define $x$ to be a large parameter, and fix $\beta$ such that $0<\beta<1$.  The function $f$ will be parametrized by $x$, and hence one considers $f_x(n)$.

We modify Shiu's first condition in the following way.  
\begin{condI}\label{Condition1} 
For any prime $p\leq x$,
\begin{gather}\label{Cond1}
f_x(p^l)\leq  (A_1(\log\log x)^{\beta})^l.
\end{gather}
\end{condI}
Obviously, this condition will depend upon the $\beta$ chosen.  In our theorems, we will have explicit bounds for $\beta$ that will depend on $\alpha$ and $\kappa$.

For the second condition, we would like to use something similar to Condition ii above.  However, this is too restrictive for our purposes, since Condition 1 would allow $f(2)$ to be as big as $(\log\log x)^{\beta}$, which would contradict the requirement that $f(2)\ll 2^\epsilon$. 

To counter this, we add an extra term to handle the case of $p$ small:
\begin{condII}\label{Condition2}
For any $\epsilon>0$, there exists an $A_2=A_2(\epsilon)$ such that for every $n\leq x$,
\begin{gather}\label{Cond2}
f_x(n)\leq \max\{A_2 n^\epsilon,A_2(\log x)^\epsilon\}.
\end{gather}
\end{condII}
Since the conditions depend on $\beta$ and $x$, we will denote by $\mathcal M(\beta,x)$ the set of nonnegative multiplicative functions that satisfy both conditions for a given $\beta$ and $x$.  We will say that a function $f$ (which depends on the parameter $x$) is in $\mathcal M(\beta)$ if it is in
$\mathcal M(\beta,x)$ for all $x$ such that, for a given choice of $\epsilon>0$, the same $A_1$ and $A_2(\epsilon)$ apply uniformly for all $x>100$.

In the theorems below, since the choice of $x$ will be clear from the context, we write $f$ for $f_x$.  The implied constants in all of the theorems are independent of $x$.

\begin{theorem}\label{Maintheorem}
Let $0<\alpha<\frac 12$, $0<\kappa<\frac 12$, and $\epsilon_0>0$.  Let $x>100$, and let $x^\kappa\leq y\leq x$ and $k<y^{1-\alpha}$.  Moreover, let $f\in \mathcal M(\beta)$ for some fixed $\beta<\frac{\alpha\kappa}{41}$.  Then
$$\sum_{\substack{x-y\leq n\leq x \\ n\equiv a\pmod k}}f(n)\ll \frac{y}{\phi(k)(\log x)^{1-\epsilon_0}}\exp\left(\sum_{\substack{p\leq x \\ p\nmid k}}\frac{f(p)}{p}\right).$$
Here, the implied constants depend on our choices of $\alpha$, $\kappa$, $\beta$, and $\epsilon_0$, as well as $A_1$ and $A_2(\epsilon)$ for a choice of $\epsilon$ that depends on $\alpha$ and $\kappa$. \end{theorem}

We remark that the $\frac{1}{(\log x)^{1-\epsilon_0}}$ in the theorem above could be improved to $\frac{e^{(\log\log x)^{1-\epsilon_0}}}{\log x}$ if we tightened Condition 2 to 
\begin{gather*}
f(n)\leq \max\{A_2 n^\epsilon,A_2e^{(\log\log x)^{1-\gamma}}\}
\end{gather*}
for some $0<\gamma<1$.

Further tightening of Condition 2 to smaller orders of magnitude may well lead to further improvement.  However, it appears that this method can not fully recover the $1/\log x$ that appears in Shiu's original argument as long as $f_x(p)$ goes to infinity as $x$ goes to infinity.  We discuss this limitation a bit more in Section \ref{outlines}.

\section{Introduction: Smooth-Supported Functions}

Next, we examine the case where $f$ is a smooth-supported function.  We will call a number $Q$-smooth if all of its prime factors are less than or equal to $Q$.  For a given $Q\leq x$, define
\begin{gather}\label{udef}
u=\frac{\log x}{\log Q}.
\end{gather}
Let $\Psi(x,Q)$ denote the number of $Q$-smooth numbers up to $x$, and let $\mathcal S(Q)$ denote the $Q$-smooth numbers.  It is known that when
$Q\geq \exp\left((\log\log x)^{\frac 53+\epsilon}\right)$, an estimate for $\Psi(x,Q)$ is given by the asymptotic formula
\begin{gather}\label{Di}\Psi(x,Q)=x\rho(u)(1+o(1)),\end{gather}
where $\rho$ is the Dickman-de Bruijn function \cite{Hi86}.  It is believed that this holds for all $Q\geq (\log x)^{2+\epsilon}$, and indeed this is known to be true under the Riemann Hypothesis \cite{Hi84}.  
Importantly, 
$$\rho(u)=u^{-u+o(1)}$$
holds for all $u$ \cite[(1.16)]{Gr}.  

On a short interval, we have a 1985 result of Hildebrand \cite{Hi85}, which states that if $Q$ is sufficiently large and $x>y>Q$ then
\begin{gather}\label{Hild}\Psi(x+y,Q)-\Psi(x,Q)\leq \Psi(y,Q).
\end{gather}
On the other hand, Shiu's theorem with $k=1$ gives a bound of 
$$\sum_{\substack{x-y\leq n\leq x \\ n\in \mathcal S(Q)}}1\ll \frac{y\log Q}{\log x}=\frac yu.$$
Shiu's bound is clearly not sharp here, as $\frac yu$ will generally be much larger than $\Psi(y,Q)$.

Here, we bring the Shiu bound a bit closer to expectations for any multiplicative function $f$.  For ease of notation, we will define $\mathcal M_Q$ to be the set of $Q$-smooth supported nonnegative functions $f$ in $\mathcal M$, and we will define $\mathcal M_Q(\beta)$ analogously.

\begin{theorem}\label{MnThm2}
Let $x$, $y$, $\alpha$, $\kappa$, and $k$ be as in Theorem \ref{Maintheorem}, let $Q\leq x$, and let $u$ be as in (\ref{udef}).  There exists a constant $C_1$ such that if $f\in \mathcal M_Q$ then
$$\sum_{\substack{x-y\leq n\leq x \\ n\equiv a\pmod k}}f(n)\ll \rho(u)^{C_1}\frac{y}{\phi(k)\log x}\exp\left(\sum_{\substack{p\leq Q \\ p\nmid k}}\frac{f(p)}{p}\right).$$
Moreover, let $Q\leq e^{\frac{\log x\log\log\log x}{\log\log x}}$.  Then there exists a $\beta<\frac{\alpha\kappa}{41}$ and a constant $C_2$ such that if $f\in \mathcal M_Q(\beta)$ then for any $\epsilon_0>0$,
$$\sum_{\substack{x-y\leq n\leq x \\ n\equiv a\pmod k}}f(n)\ll \rho(u)^{C_2}\frac{y}{\phi(k)(\log x)^{1-\epsilon_0}}\exp\left(\sum_{\substack{p\leq Q \\ p\nmid k}}\frac{f(p)}{p}\right).$$
\end{theorem}
It will suffice to take 
$$C_1=\frac{\alpha\kappa}{41}$$ 
and
$$C_2=\frac{5}{656}\left(\frac{\alpha\kappa \left(1-\frac{\alpha \kappa}{20}\right)}{20}\right).$$

We note here that while we do not put a lower bound on $Q$, Theorem \ref{MnThm2} ceases to be non-trivial when $Q=(\log y)^K$ when $K$ is roughly of size $\frac{\log y}{(1-C_i)\log k}$.  This is because such a $Q$ would have $\rho\left(\frac{\log y}{\log Q}\right)\ll y^{-\frac{1}{K}+\epsilon}$, and hence using (\ref{Hild}) with the trivial bound of $f(n)\ll y^\epsilon$ gives a bound of
$$\sum_{\substack{x-y\leq n\leq x \\ n\equiv a\pmod k}}f(n)\ll y^\epsilon\sum_{\substack{x-y\leq n\leq x \\ n\equiv a\pmod k}}1\ll y^{1-\frac 1K+\epsilon},$$
while our bound would give
$$\sum_{\substack{x-y\leq n\leq x \\ n\equiv a\pmod k}}f(n)\ll \frac{y^{1-\frac{C_i}K}}{\phi(k)(\log x)^{1-\epsilon}}\exp\left(\sum_{p\leq x}\frac{f(p)}{p}\right),$$
which is roughly equal (depending on the sum).

\section{Applications}

We also present two applications of these theorems.  We begin with a result on smooth numbers in short intervals, noting that one of the inspirations for this paper comes from a 2008 result of Soundararajan on smooth numbers \cite{So}.  We begin by stating Soundararajan's result here.

\begin{theorem}[Soundararajan 2008]\label{SoundThm}
Assume the Riemann Hypothesis. Let $x$ be large, and suppose that $Q\geq \exp\left(5\sqrt{\log x\log\log x}\right)$.  There is an absolute constant $B$ such that with $y = Bu\sqrt x/\rho(u/2)$ we have
\begin{gather}\label{Soundbound}
\Psi(x+y,Q)-\Psi(x,Q)\gg_\epsilon yx^{-\epsilon}
\end{gather}
for any $\epsilon>0$.
\end{theorem}
More recently, Goudout \cite{Go} found an unconditional version of this theorem, albeit with a smaller range for $Q$ and $y$ but a better lower bound:
\begin{theorem}[Goudout 2017]\label{GoThm}
Fix a positive $\epsilon<\frac 16$ and a constant $u_0$ depending on $\epsilon$. Let $Q$ be such that $\exp\left((\log x)^{\frac 56+\epsilon}\right)\leq Q\leq x^{\frac 1{u_0}}$.  Then for any $h$ with $\rho(u)^{-3-\epsilon}\leq h\leq \sqrt x$, if we let $y=h\sqrt x$ then
\begin{gather}\label{Gobound}
\Psi(x+y,Q)-\Psi(x,Q)\gg y\frac{\rho(u)^2}{\log^3 x}.
\end{gather}
\end{theorem}
With our improvement on Shiu's theorem, we can improve the lower bound in (\ref{Soundbound}) to almost the correct order of magnitude, and we can similarly improve the exponent on $\rho(u)$ in (\ref{Gobound}) to $1+o(1)$:
\begin{theorem}\label{smooththm}
Let $Q$, $x$, and $y$ satisfy either the hypotheses in either Theorem \ref{SoundThm} or \ref{GoThm}, where in the former case we also assume the Riemann Hypothesis.  Additionally, assume that $u\geq (\log\log x)^2.$  Then there exists a constant $C_3$ such that 
\begin{gather}
\Psi(x+y,Q)-\Psi(x,Q)\gg y\rho(u)^{1+\frac{C_3\log\log \log\log x}{\log\log \log x}}.
\end{gather}
In particular, we can take $C_3>1-C_2$.
\end{theorem}
The assumption here that $u\geq (\log\log x)^2$ is simply to eliminate log factors, since if $u$ satisfies this inequality then for any $A$, $$\log^A x=\rho(u)^{o(1)},$$
and hence any log-terms can be absorbed into the $\rho$-term.  For smaller $u$, it is not clear that our result would improve on Goudout's.  However, the more interesting applications tend to have larger $u$'s, so this lower bound for $u$ is a relatively small price to pay.

We note that Jain \cite{Ja} has proven a lower bound for $Q$-smooth numbers in short intervals with $Q$ as small as $$\exp\left((\log x)^{\frac 23}(\log\log x)^{\frac 43+\epsilon}\right)$$ and $$y\geq \exp\left((1+\epsilon)\left(\frac{11}{16}u\log u+2\log\log x\right)\right).$$
However, the methods in that paper do not currently seem amenable to improvement by our methods, and hence we do not pursue this further.

In order to prove this result, we give also a result on $\tau_d$, the $d$-fold divisor function (where $\tau=\tau_2$), that may be of independent interest. Here, we find a bound for the sum of $\tau_d(n)^R$ as $R$ slowly goes to infinity:

\begin{theorem}\label{dfold}
Let $x$, $y$, and $k$ be as in Theorem \ref{Maintheorem}.  Let $d\geq 2$, and let $R=\frac{\log\log\log x}{g(x)}$ for some $g(x)$ such that $g(x)\to\infty$ as $x\to\infty$.  Then for any $\epsilon>0$,
$$\sum_{\substack{x-y\leq n\leq x \\ n\equiv a\pmod k}}\tau_d(n)^{R}\ll \frac{y(\log x)^\epsilon}{\phi(k)}\left(\frac{\phi(k)}{k}\log x\right)^{d^{R}-1}.$$
Moreover, let $Q\leq e^{\frac{\log x\log\log\log x}{\log\log x}}$.  Then
$$\sum_{\substack{x-y\leq n\leq x \\ n\equiv a\pmod k \\ n\in \mathcal S(Q)}}\tau_d(n)^{R}\ll \rho(u)^{C_2+o(1)}\frac{y(\log x)^\epsilon}{\phi(k)}\left(\frac{\phi(k)}{k}\log x\right)^{d^{R}-1}.$$
\end{theorem}

We note that while Condition 1$\beta$ was chosen to maximize the allowable values for $f(p)$, Condition 2 was chosen to maximize the $R$ for which we could bound $\tau(n)^{R}$.  For any $p$,
$$\tau(p)^{R}\leq 2^R,$$
and so Condition 1$\beta$ is satisfied as long as $R<\beta \log\log \log x/\log 2$.  For Condition 2, however, there is a sharp upper bound for $\tau$ given in \cite{Wig} as
$$\tau(n)\leq e^{\frac{C\log n}{\log\log n}}$$
for some $C$.  So we require that for every $\epsilon>0$,
$$\tau(n)^R\leq e^{\frac{C\log nR}{\log\log n}}\leq \max\{A_2n^\epsilon,A_2(\log x)^\epsilon\}.$$
Taking the latter term in the maximum, we require
$$\frac{CR\log n}{\log\log n}\leq \log(A_2)+\epsilon(\log \log x),$$
and since the latter term in the maximum is the larger term only on the interval where $n\leq \log x$, we require
$$\frac{CR\log \log x}{\log\log \log x}\leq \log(A_2)+\epsilon(\log \log x).$$
Hence, we need
$$R\leq \epsilon'(\log \log \log x)$$
for every $\epsilon'>0$.


\section{Outlines of proofs}\label{outlines}
\subsection{Outline of Proof of Theorem \ref{Maintheorem}}

Shiu's original proof relies heavily on the following identity: for any $\delta>3/4$,
\begin{align*}
\sum_{\substack{n\leq x \\ (n,k)=1}}\frac{f(n)}{n^\delta}\leq &\prod_{\substack{p\leq x \\ p\nmid k}}\left(1+\frac{f(n)}{n^\delta}+\sum_{l=2}^\infty \frac{f(p^l)}{p^{l\delta}}\right)\leq \exp\left(\sum_{\substack{p\leq x \\ p\nmid k}}\left(\frac{f(p)}{p^\delta}+\sum_{l=2}^\infty \frac{f(p^l)}{p^{l\delta}}\right)\right).
\end{align*}
Since $f(p)$ is bounded by a constant $A_1$ in Shiu's work, we have that if $p>A_1^2$ then
\begin{gather}\label{psum}\sum_{l=2}^\infty \frac{f(p^l)}{p^{l\delta}}\leq \sum_{l=2}^\infty \left(\frac{A_1}{p^\delta}\right)^l\leq \frac{2A_1^2}{p^{2\delta}}\end{gather}
Meanwhile, if $p^2\leq A_1$ then (by Condition ii) we have $f(p^l)\leq A_2p^{l\epsilon}$, which means that 
\begin{gather}\label{psum2}\sum_{l=2}^\infty \frac{f(p^l)}{p^l}\leq \sum_{l=2}^\infty A_2\left(\frac{p^\epsilon}{p^\delta}\right)^l\leq \frac{3A_2}{p^{2\delta-2\epsilon}}.\end{gather}
Hence 
\begin{gather}\label{psum3}
\sum_p \sum_{l=2}^\infty \frac{f(p^l)}{p^{l\delta}}\ll 1,
\end{gather}
and thus
\begin{align}\label{psum3simple}
\sum_{\substack{n\leq x \\ (n,k)=1}}\frac{f(n)}{n^\delta}\ll \exp\left(\sum_{\substack{p\leq x \\ p\nmid k}}\frac{f(p)}{p^\delta}\right).
\end{align}
Shiu then splits the sum 
$$\sum_{\substack{x-y\leq n\leq x \\ n\equiv a\pmod k}}f(n)$$
into four subsums.  Three of the sums can be handled by sieve arguments, bounds on smooth numbers, or similar.  The arguments Shiu uses for these three sums can be applied here essentially unchanged.

For the fourth sum, however, Shiu uses an alteration to (\ref{psum3simple}), finding that for relatively small $r$ and for $z$ with $z\to \infty$,
\begin{align}\label{Lemma3}
\sum_{\substack{n\geq z^\frac 12 \\ p(n)\leq z^\frac 1r \\ (n,k)=1}}\frac{f(n)}{n}\ll \exp\left(\sum_{\substack{p\leq z \\ p\nmid k}}\frac{f(p)}{p}-\frac{1}{10}r\log r\right).
\end{align}
This comes in handy when $n$ has a ``reasonable" number of distinct prime factors $r$, which is the set of $n$ that Shiu denoted by IV in his paper.  In this case, the negative term will counteract the growing number of factors.  In particular, for an $r_0$ depending on $x$, he is able to bound the sum of $f(n)$ in this set with
\begin{align*}
\sum_{n\in IV}f(n)\ll &\sum_{2\leq r\leq r_0}\frac{y}{\phi(k)\log x}A_1^{\frac{20r}{\alpha\kappa}}\exp\left(\sum_{\substack{p\leq z \\ p\nmid k}}\frac{f(p)}{p}-\frac{1}{10}r\log r\right)\\
\ll &\frac{y}{\phi(k)\log x}\exp\left(\sum_{\substack{p\leq z \\ p\nmid k}}\frac{f(p)}{p}\right),
\end{align*}
which yields the main theorem in Shiu's paper.

If $f_x(p)$ is an increasing function on $x$, however, a few of the steps in the bound for IV fall apart.  For one, (\ref{psum3}) no longer holds, since $A_1$ and $A_2$ in (\ref{psum}) and (\ref{psum2}) are replaced by functions of $x$.  The bound in (\ref{Lemma3}) also no longer holds for much the same reason.  Moreover, in the bound for the sum over IV, $A_1^{\frac{20r}{\alpha\kappa}}$ is replaced by $\max_{p\leq x}\left(A_1 f(p)\right)^\frac{20r}{\alpha\kappa}$, which is a function of $r$ that may not be smaller than $\exp\left(\frac{1}{10}r\log r\right)$.  

In order to close these gaps in the proof, we make a couple of alterations.  A key change is that, for a parameter $J=(\log \log x)^{1-\beta}$, we split $n$ into $J$-smooth and $J$-rough parts, which we will call $s$ and $t$, respectively.   Since $f$ is multiplicative, we have $f(n)=f(s)f(t)$.  The sum over $s$ is small and relatively easy to handle, since $s$ itself is usually small.  For the sum over $t$, we now have
\begin{gather}\label{psum4}\sum_{p\geq J}\sum_{l=2}^\infty \frac{f(p^l)}{p^{l\delta}}\ll \sum_{p\geq J}\sum_{l=2}^\infty \left(\frac{A_1(\log\log x)^\beta}{p^\delta}\right)^l\leq 2\sum_{p\geq J}\frac{A_1^2(\log\log x)^{2\beta}}{p^{2\delta}}\ll \frac{A_1^2(\log\log x)^{2\beta}}{J^{2\delta-1}}=o(1),\end{gather}
since $\beta$ is much smaller than 1/5.  Hence,
\begin{align}\label{psum4simple}
\sum_{\substack{n\leq x \\ (n,k)=1 \\ n\in \mathcal T}}\frac{f(n)}{n^\delta}\ll \exp\left(\sum_{\substack{J<p\leq x \\ p\nmid k}}\frac{f(p)}{p^\delta}\right).
\end{align}
In the cases where we have a larger number $r$ of distinct prime factors (i.e. the case where $n\in IV$), we will now have a slightly worsened bound of
\begin{align}\label{Lemma3a}\sum_{n\in IV}f(n)\ll &\sum_{2\leq r\leq r_0}\frac{y}{\phi(k)\log x}(A_1 (\log\log x)^{\beta})^\frac{20r}{\alpha\kappa}\\
&\cdot\exp\left(\sum_{\substack{p\leq x \\ p\nmid k}}\frac{f(p)}{p}-\frac 12(1-\frac 54\beta)r\log r+2A_1(\log\log x)^\beta r^{1-\frac{5}{4}\beta}\right).
\end{align}
The $\frac 12(1-\frac 54\beta)r\log r$ will eventually dominate once $r$ is large enough, but this will not occur until $r$ reaches roughly $(\log\log x)^{1-\frac{\beta}{2}}$.  This transition period from $r=2$ to $r$ being close to $(\log\log x)^{1-\frac{\beta}{2}}$ is one of the reasons for the loss of $(\log x)^\epsilon$ in the main theorem.

We remark that the computation of Class IV is actually where the upper bound for $\beta$ arises.  In particular, if we take $r$ to be bounded below by $\log \log x$ (as it is at the end of the aforementioned transition period), we see that
\begin{align*}
&(A_1 (\log\log x)^{\beta})^\frac{20r}{\alpha\kappa}\exp\left(-\frac 12(1-\frac 54\beta)r\log r\right)\\
&\ll \exp\left(r\log\log \log x\left(\frac{20\beta}{\alpha\kappa}-\frac 12(1-\frac 54\beta)\right)\right).
\end{align*}
In order for the sum over $r$ to remain harmless, this term must be negative.

Additionally, we mentioned in Section \ref{larger} that while tightening Condition 2 to smaller orders of magnitude could lead to further improvement, the current methods seemed unable to recover the $1/\log x$ that appears in Shiu's original argument as long as $f_x(p)$ can grow to infinity as $x$ does the same.  We will see in the proof of Lemma \ref{V3bd} that there is a built-in limitation that the bound coming from this method must be at least 
$$\frac{yH^{1+\epsilon}}{\phi(k)(\log x)}\exp\left(\sum_{\substack{p\leq x \\ p\nmid k}}\frac{f(p)}{p}\right),$$
where $H$ is the largest value where one can have $f(H)\gg H$.  More specifically, in Section \ref{removingsmooth}, we have sums of the form
$$\sum_{s}\frac{f(s)}{s}$$
where the $s$ are $J$-smooth numbers.  We bound $f(s)$ by $\max\{A_2n^\epsilon,A_2(\log x)^\epsilon\}$ and then sum over the remaining expression.  On the interval where $s$ is quite a bit larger than $\log x$ (we use $s>(\log x)^{10}$), the sum will be small because one has
$$\sum_{s>(\log x)^{10}}\frac{f(s)}{s}\ll \sum_{s>(\log x)^{10}}\frac{1}{s^{1-\epsilon}},$$
and the terms in this sum are sparse in that $\Psi(z,J)=o(z^\epsilon)$ when $z>(\log x)^{10}$.  So this sum will be bounded by a constant.

On the other hand, when $s<(\log x)^{10}$, we have
$$\sum_{s<(\log x)^{10}}\frac{f(s)}{s}\ll (\log x)^\epsilon\sum_{s<(\log x)^{10}}\frac{1}{s}\ll (\log x)^\epsilon\log\log x.$$
This is the source of the extra power of $\log x$.  Tightening Condition 2 to something smaller than $A_2(\log x)^\epsilon$ would improve this bound, but we could not recover the full $1/\log x$ without removing this part of the condition entirely.

\subsection{Outline of Proof of Theorem \ref{MnThm2}}

The case where $f$ is smooth-supported is a bit simpler, since if $u$ is sufficiently large with respect to $x$ then it is not possible for $r$ to be small.  After all, if $f(n)$ is only supported on $Q$-smooth numbers then $n$ must have at least $\frac{\log x}{\log Q}$ factors.  Hence we can use either (\ref{Lemma3}) if $f\in M_Q$ or (\ref{Lemma3a}) with $r\geq \frac{\log x}{\log Q}$ if $f\in M_Q(\beta)$.

\section{Removing Smooth Numbers}\label{removingsmooth}

Now, we begin the proof of Theorem \ref{Maintheorem}.   We first show that the contribution of numbers $n$ with a large smooth factor will be minimal.

As above, let $$J=(\log \log x)^{1-\beta}.$$
Define as usual the $J$-rough numbers to be those numbers $n$ where all of the primes $p|n$ are such that $p>J$.  We define the following notation:
\begin{gather*}
\mathcal S=\{s\in \mathbb N:s\mbox{ is }J\mbox{-smooth}\},\\
\mathcal T=\{t\in \mathbb N:t\mbox{ is }J\mbox{-rough}\}.
\end{gather*}
Our decomposition of $n$ is now $n=st$ where $s\in \mathcal S$ and $t\in \mathcal T$.  (In fact, we will generally use the convention that $s$ and $t$ indicate numbers in $\mathcal S$ and $\mathcal T$, respectively.)  For ease of notation, we also define
\begin{gather*}
\mathcal S[x,y]=\{s\in [x,y]:s\in \mathcal S\},\\
\mathcal T[x,y]=\{t\in [x,y]:t\in \mathcal T\},
\end{gather*}
and we will have an analogous definition for $\mathcal S(x,y]$ or $\mathcal T(x,y]$.
Then
\begin{align*}
&\sum_{\substack{x-y\leq n\leq x \\ n\equiv a\pmod k}}f(n)\\
&=\sum_{s\in \mathcal S[1,(\log x)^{10}]}\sum_{\substack{t\in \mathcal T[\frac xs-\frac ys,\frac xs] \\ t\equiv a\overline{s}\pmod k}}f(st)+\sum_{s\in \mathcal S((\log x)^{10},x^\alpha]}\sum_{\substack{t\in \mathcal T[\frac xs-\frac ys,\frac xs] \\ t\equiv a\overline{s}\pmod k}}f(st)+\mathop{\sum\sum}_{\substack{x\leq st\leq x+y \\ st\equiv a\pmod k \\ s\in \mathcal S(x^\alpha,x],\mbox{ }t\in \mathcal T}}f(st)\\
&=V_1+V_2+V_3,
\end{align*}
say.  

We can bound $V_3$ immediately.

\begin{lemma}\label{LemmaV3} For $V_3$ as above,
$$V_3=o\left(\frac{y}{kx^{\frac{\alpha}{3}}}\right).$$
\end{lemma}
\begin{proof}
We first state a bound in a bit more generality than we will need for this lemma because we will use this bound a few times in different contexts.  For any $z$ such that $z\geq \log x$, we have $J\leq (\log z)^{1-\beta}$, and thus we use the estimate from \cite[(1.18)]{Gr} that
$$\Psi(z,J)=\left(\frac{\log z}{J}\right)^{(1+o(1))\pi(J)}.$$
Using the bounds that 
$$\pi(J)\leq \frac{2J}{\log J},$$
and trivially $$\frac{\log z}{J}\leq \log z,$$
we have
\begin{gather}\label{psizJ}
\Psi(z,J)\leq \left(\log z\right)^{(2+o(1))\frac{J}{\log J}}\leq z^{\frac{\log \log z}{(\log z)^\beta\log J}}=o\left(z^\epsilon\right)
\end{gather}
for any $\epsilon>0$ .

Now, note that if $s$ is $J$-smooth and $s>x^\alpha$ then $s$ can be split into $s=s_1s_2$, where $s_1\in\mathcal S[x^{\frac{\alpha}{2}},x^{\frac{\alpha}{2}}J]$.  (This decomposition is not necessarily unique, but this will not matter.)  Write $n'=s_2t$.  Recalling that $f(n)\leq A_2n^\epsilon$ when
$n>\log x$, we have
\begin{align*}
V_3\leq &A_2x^{\epsilon}\sum_{s_1\in\mathcal S[x^{\frac{\alpha}{2}},x^{\frac{\alpha}{2}}J]}\sum_{\substack{\frac x{s_1}\leq n'\leq \frac{x+y}{s_1} \\ n'\equiv a\overline{s_1}\pmod k}}1\\
\leq &2A_2x^{\epsilon}\sum_{s_1\in\mathcal S[x^{\frac{\alpha}{2}},x^{\frac{\alpha}{2}}J]}\frac{y}{ks_1}\\
\leq &2A_2 \frac{y}{kx^{\frac{\alpha}{2}}}\sum_{s_1\in\mathcal S[x^{\frac{\alpha}{2}},x^{\frac{\alpha}{2}}J]}1.
\end{align*}
By (\ref{psizJ}), we know that 
$$\Psi(x^\frac \alpha 2 J,J)=o\left(x^\epsilon\right)$$
for any $\epsilon>0$.

Taking $\epsilon=\alpha/6$, we then have
$$V_3=o\left(\frac{y}{kx^{\frac{\alpha}{3}}}\right).$$
\end{proof}

In order to bound the other two sums, we will require the following theorem, which we will prove later.  
\begin{theorem}\label{roughbd}
Let $f\in M(\beta)$.  Then for any $\epsilon_0$ with $0<\epsilon_0<1$,
$$\sum_{\substack{t\in \mathcal T[x-y,x]\\ t\equiv a\pmod k}}f(t)\ll \frac{y}{\phi(k)(\log x)^{1-\epsilon_0}}\exp\left(\sum_{\substack{J<p\leq x \\ p\nmid k}}\frac{f(p)}{p}\right).$$
Here, the constant is independent of $x$.
\end{theorem}

We also prove a variant for smooth-supported functions:
\begin{theorem}\label{roughbd2}
Let $f\in M_Q(\beta)$. Then there exists a $C_2$ such that for any $\epsilon_0$ with $0<\epsilon_0<1$,
\begin{gather}\label{rbdpart1}\sum_{\substack{t\in \mathcal T[x-y,x]\\ t\equiv a\pmod k}}f(t)\ll \frac{y\rho(u)^{C_2+o(1)}}{\phi(k)(\log x)^{1-\epsilon_0}}\exp\left(\sum_{\substack{J<p\leq x \\ p\nmid k}}\frac{f(p)}{p}\right).
\end{gather}
Here also, the constant is independent of $x$.
\end{theorem}

%

If we are willing to assume these, we can easily find bounds for $V_1$ and $V_2$.  For ease of notation, define $\varrho$ such that $\varrho=1$ if $f\in M(\beta)$ and $\varrho=\rho(u)^{C_2+o(1)}$ if $f\in M_Q(\beta)$.

\begin{lemma}\label{V3bd}
Assume Theorem \ref{roughbd} and \ref{roughbd2}.  Then for any $\epsilon_1$ with $0<\epsilon_1<1$,
$$V_1\ll \frac{y\varrho}{\phi(k)(\log x)^{1-\epsilon_1}}\exp\left(\sum_{\substack{J<p\leq x \\ p\nmid k}}\frac{f(p)}{p}\right).$$
\end{lemma}
\begin{proof}
We know first that
$$\sum_{s\in \mathcal S[1,(\log x)^{10}]}\frac 1s\leq 11\log\log x+2.$$
So applying Condition \ref{Condition2} and either theorem for some $\epsilon_1>\epsilon_0$ gives
\begin{align*}
\sum_{s\in \mathcal S[1,(\log x)^{10}]}\sum_{\substack{t\in \mathcal T[\frac xs-\frac ys,\frac xs] \\ t\equiv a\overline{s}\pmod k}}f(st)\ll & \sum_{s\in \mathcal S[1,(\log x)^{10}]}\frac{y\varrho f(s)}{\phi(k)s(\log x)^{1-\epsilon_0}}\exp\left(\sum_{\substack{J<p\leq x \\ p\nmid k}}\frac{f(p)}{p}\right)\\
\ll & \frac{y\varrho }{\phi(k)(\log x)^{1-\epsilon_1}}\exp\left(\sum_{\substack{J<p\leq x \\ p\nmid k}}\frac{f(p)}{p}\right),
\end{align*}
which is as required.
\end{proof}
Finally, we handle $V_2$.
\begin{lemma}
Assume Theorem \ref{roughbd} or \ref{roughbd2}.  Then
$$V_2\ll \frac{y\varrho }{\phi(k)(\log x)^{9}}\exp\left(\sum_{\substack{J<p\leq x \\ p\nmid k}}\frac{f(p)}{p}\right).$$
\end{lemma}
\begin{proof}
We begin with the assumed theorem:
\begin{gather}\label{T2sum}\sum_{s\in \mathcal S((\log x)^{10},x^\alpha]}\sum_{\substack{t\in \mathcal T[\frac xs-\frac ys,\frac xs] \\ t\equiv a\overline{s}\pmod k}}f(st)\ll \sum_{s\in \mathcal S((\log x)^{10},x^\alpha]}\frac{y\varrho f(s)}{\phi(k)s(\log x)^{1-\epsilon_0}}\exp\left(\sum_{\substack{J<p\leq x \\ p\nmid k}}\frac{f(p)}{p}\right).\end{gather}
Note that for any $s\in \mathcal S((\log x)^{10},x^\alpha]$, we have 
$$(\log\log x)^{1-\beta}\leq \left(\frac{1}{10}\log s\right)^{1-\beta}.$$
So by (\ref{psizJ}),
\begin{gather}\label{zpsi}
\Psi(s,J)=o\left(s^{\epsilon}\right)
\end{gather}
for any $\epsilon>0$.

So we apply partial summation to the sum $V_2$, where we will use the variable $v$ to indicate where we have dropped the requirement of smoothness:
\begin{align*}\sum_{s\in \mathcal S((\log x)^{10},x^\alpha]}\frac{f(s)}{s}=&\sum_{v\in \mathcal S((\log x)^{10},x^\alpha-1]}\left(\frac{1}{v}-\frac{1}{v+1}\right)\sum_{s\in \mathcal S((\log x)^{10},v]}f(s) +\frac{1}{x^{\alpha}}\sum_{s\in \mathcal S((\log x)^{10},x^\alpha]}f(s).
\end{align*}
Applying (\ref{zpsi}) to the sum and $f(v)\leq A_2v^\epsilon$ to the function for some $\epsilon<\alpha/10$ then gives
\begin{align*}\sum_{s\in \mathcal S((\log x)^{10},x^\alpha]}\frac{f(s)}{s}&\leq A_2\left(\sum_{v\in ((\log x)^{10},x^\alpha -1]}\frac{1}{v^{2-2\epsilon}} \right)+\frac{A_2}{x^{\alpha-2\epsilon}}
\ll (\log x)^{-9}.
\end{align*}
Plugging this into (\ref{T2sum}) then yields the lemma.

\end{proof}

Putting our estimates for $V_1$, $V_2$, and $V_3$ together then proves Theorem \ref{Maintheorem} and the second half of Theorem \ref{MnThm2}.

The next sections will then set up the machinery to prove Theorems \ref{roughbd} and \ref{roughbd2}.
\section{The Sum over Rough Numbers}

Hereafter, for ease of notation, we will write $X=x/s$, $Y=y/s$, and $b\equiv a\overline{s}\pmod k$, where $k$ obeys the bounds in Theorems \ref{Maintheorem} and \ref{MnThm2}.  We are tasked with evaluating
$$\sum_{\substack{t\in \mathcal T[X-Y,X] \\ t\equiv b\pmod k}}f(t).$$

From here, our methods will largely follow those of \cite{Sh}.  As in that paper, we will let $p(n)$ and $q(n)$ denote the greatest and least prime factors of $n$, respectively, and we define
$$\Phi(x,y,z;k,a)=\sum_{\substack{x-y\leq n\leq x \\ n\equiv a \pmod k \\ q(n)>z}}1.$$
By Lemma 2 of \cite{Sh}, we know that 
\begin{gather}\label{rough}\Phi(x,y,z;k,a)\ll \frac{y}{\phi(k)\log z}+z^2.\end{gather}
Later, we will define $z$ such that $z$ is at most $(y/k)^{1/10}$, and hence the latter term can be ignored.

We also record a few statements about smooth numbers for later.  From \cite[Lemma 1]{Sh},
\begin{gather}\label{psibound}
\Psi(x,\log x\log\log x)\ll \exp\left(\frac{3\log x}{(\log \log x)^\frac 12}\right).
\end{gather}
Meanwhile, for larger $z$, we know from (\ref{Di}) that
\begin{gather}\label{smoothbd}
\Psi(x,z)\ll x\exp\left(-\frac{\log x(\log x-\log z)}{\log z}\right).
\end{gather}
We will also require one more preliminary lemma before we turn our attention to the main sum.  This is a slightly altered version of \cite[Lemma 4]{Sh}.

\begin{lemma}\label{lemmalogr}
Let $f\in \mathcal M(\beta)$.  Then as $z\to\infty$,
$$\sum_{\substack{n\geq z^\frac 12 \\ p(n)\leq z^\frac 1r \\ (n,k)=1 \\ n\in \mathcal T}}\frac{f(n)}{n}\ll \exp\left(-\frac 12\left(1-\frac{5}{4}\beta\right)r\log r+\sum_{\substack{J<p\leq z^\frac 1r \\ p\nmid k}}\frac{f(p)}{p}+2A_1 (\log\log x)^{\beta}r^{1-\frac{5}{4}\beta}\right)$$
uniformly in $k$ and $r$, provided that $1\leq r\leq \frac{\log z}{4\log\log z}$.
\end{lemma}
\begin{proof}
Let $\frac 34 <\delta\leq 1$.  By Rankin's trick,
\begin{align*}
\sum_{\substack{n\geq w \\ p(n)\leq v \\ (n,k)=1 \\ n\in \mathcal T}}\frac{f(n)}{n}\leq &w^{\delta-1}\sum_{\substack{n\geq w \\ p(n)\leq v \\ (n,k)=1 \\ n\in \mathcal T}}\frac{f(n)}{n^\delta}\leq w^{\delta-1}\sum_{\substack{n\geq 1 \\ p(n)\leq v \\ (n,k)=1 \\ n\in \mathcal T}}\frac{f(n)}{n^\delta}\\
=&w^{\delta-1}\prod_{\substack{J<p\leq v \\ p\nmid k}}\left(1+\frac{f(p)}{p^\delta}+\sum_{l=2}^\infty \frac{f(p^l)}{p^{l\delta}}\right)\ll \exp\left((\delta-1)\log w+\sum_{\substack{J<p\leq v \\ p\nmid k}}\frac{f(p)}{p^\delta}\right)
\end{align*}
by (\ref{psum4simple}), where $v$ and $w$ are parameters to be chosen later.

We rewrite $\frac{f(p)}{p^\delta}$ as $\frac{f(p)}{p}+\frac{f(p)}{p}(p^{1-\delta}-1)$ and note that
\begin{align*}
\sum_{\substack{J<p\leq v \\ p\nmid k}}\frac{f(p)}{p}(p^{1-\delta}-1)\leq &A_1 (\log\log x)^{\beta}\sum_{p\leq v}\frac 1p\sum_{n=1}^\infty \frac{\left((1-\delta)(\log p)\right)^n}{n!}\\
\leq &A_1 (1-\delta)(\log\log x)^{\beta}\sum_{n=1}^\infty \frac{\left((1-\delta)(\log v)\right)^{n-1}}{n!}\sum_{p\leq v}\frac{\log p}p\\
\leq &A_1 (\log\log x)^{\beta}\sum_{n=1}^\infty \frac{\left((1-\delta)(\log v)\right)^{n}}{n!}\\
\leq &2A_1 (\log\log x)^{\beta}v^{1-\delta}.
\end{align*}
So
\begin{align*}
\sum_{\substack{n\geq w \\ p(n)\leq v \\ (n,k)=1 \\ n\in \mathcal T}}\frac{f(n)}{n}\ll \exp\left((\delta-1)\log w+\sum_{\substack{J<p\leq v \\ p\nmid k}}\frac{f(p)}{p}+2A_1 (\log\log x)^{\beta}v^{1-\delta}\right).
\end{align*}

Letting $w=z^\frac 12$, $v=z^\frac 1r$, and $\delta=1-\frac{(1-\frac 54\beta )r\log r}{\log z}$ then gives

\begin{align*}
\sum_{\substack{n\geq w \\ p(n)\leq v \\ (n,k)=1 \\ n\in \mathcal T}}\frac{f(n)}{n}\ll \exp\left(-\frac 12\left(1-\frac 54\beta \right)r\log r+\sum_{\substack{J<p\leq z^\frac 1r \\ p\nmid k}}\frac{f(p)}{p}+2A_1 (\log\log x)^{\beta}r^{1-\frac 54\beta }\right).
\end{align*}
\end{proof}

\section{Partitioning the Sum: The Proof of Theorem \ref{roughbd}}
Now, we partition the sum as in Section 5 of \cite{Sh}.  Let 
$$z=Y^{\frac{\alpha}{10}}.$$
Note that by this definition, $Y/k>z^2$, and hence we can omit the $z^2$-term in (\ref{rough}).

For each $n$ with $X\leq n<X+Y$, $n\equiv b\pmod k$, write
$$n=p_1^{u_1} \cdots  p_j^{u_j}p_{j+1}^{u_{j+1}}\cdots p_l^{u_l}=c_nd_n$$
where $p_1<p_2<\cdots <p_l$ and $c_n$ is the product of the first $j$ terms, where $j$ is chosen such that
$$c_n\leq z<c_np_{j+1}^{u_{j+1}}.$$
We then define the following four subclasses of $\mathcal T$:
\begin{gather*}
\mbox{I}:q(d_n)>z^\frac 12,\\
\mbox{II}:q(d_n)\leq z^\frac 12,c_n\leq z^\frac 12,\\
\mbox{III}:q(d_n)\leq \left(\log x\log \log x\right)^5,c_n>z^\frac 12,\\
\mbox{IV}:\left(\log x\log \log x\right)^5<q(d_n)\leq z^\frac 12,c_n>z^\frac 12.
\end{gather*}
\subsection{Class I}
Here, we have
$$\sum_{n\in I}f(n)\leq \sum_{\substack{c\in \mathcal T[1,z] \\ (c,k)=1
}}f(c)\sum_{\substack{\frac{X-Y}c\leq d\leq \frac{X}{c} \\ d\equiv b\overline{c}\pmod k \\ q(d)>z^\frac 12}}f(d).$$
Note that for any such $n$, $q(d)>z^\frac 12=Y^{\frac{\alpha}{20}}=X^{\frac{\alpha\kappa}{20}}.$  So $d$ can have at most $\frac{20}{\alpha\kappa}$ prime factors, and hence $f(d)\leq \left(A_1(\log\log x)^\beta\right)^{\frac{20}{\alpha\kappa}}$.  By Lemma 2 of \cite{Sh}, we know that the number of $z^\frac 12$-rough numbers up to $x$ in a congruence class mod $k$ is $\ll \frac{y}{\phi(k)\log z}$.  Recalling that $\beta<\frac{\alpha\kappa}{41}$, we can bound the above with
\begin{align*}
\sum_{n\in I}f(n)\leq &\left(A_1(\log\log x)^\beta\right)^{\frac{20}{\alpha\kappa}}\sum_{\substack{c\in \mathcal T[1,z] \\ (c,k)=1}}f(c)\sum_{\substack{\frac{X-Y}c\leq d\leq \frac{X}{c} \\ d\equiv b\overline{c}\pmod k \\ q(d)>z^\frac 12}}1\\
\ll &\frac{Y}{\phi(k)\log z}\left(\log\log x\right)^{\frac{20}{41}}\sum_{\substack{c\in \mathcal T[1,z] \\ (c,k)=1}}\frac{f(c)}{c}\\
\ll &\frac{Y}{\phi(k)\log x}\left(\log\log x\right)^{\frac{20}{41}}\exp\left(\sum_{\substack{J<p\leq X \\ p\nmid k}}\frac{f(p)}{p}\right),
\end{align*}
where the last line is by (\ref{psum4simple}).  


\subsection{Class II}
In this case, there must exist a $p|n$ such that $p\leq z^\frac 12$ and $p^u>z^\frac 12$.  For such a $p$, let $u_p$ denote the least integer such that $p^{u_p}>z^\frac 12$.  Then $p^{-u_p}\leq \min\{z^{-\frac 12},p^{-2}\}$.  Hence
$$\sum_{p\leq z^\frac 12}\frac{1}{p^{u_p}}\leq \sum_{p\leq z^\frac 14}z^{-\frac 12}+\sum_{p>z^\frac 14}\frac 1{p^2}\ll z^{-\frac 14}.$$
Then for any $\epsilon>0$,
$$\sum_{n\in II}f(n)\ll x^\epsilon\sum_{p\leq z^\frac 12}\sum_{\substack{X\leq n<X+Y \\ n\equiv b\pmod k \\ p^{u_p}|n }}1\ll x^\epsilon\sum_{p\leq z^\frac 12}\frac{Y}{kp^{u_p}}\ll \frac{Yx^\epsilon}{kz^{\frac 14}}.$$
\subsection{Class III}
In this case, we have $c$ such that $z^\frac 12<c\leq z$ and $c$ is $(\log x\log \log x)^5$-smooth.  So 
$$\sum_{n\in III}f(n)\leq \sum_{\substack{z^\frac 12<c\leq z \\ p(c)\leq (\log x\log \log x)^5}}f(c)\sum_{\substack{\frac{X-Y}c\leq d\leq \frac{X}{c} \\ d\equiv b\overline{c}\pmod k }}f(d)\ll \sum_{\substack{z^\frac 12<c\leq z \\ p(c)\leq (\log x\log \log x)^5}}\frac{Yx^\epsilon}{kc}$$
Let $$v=\left(\frac{\log z^\frac 12}{5\log\left(\log x\log\log x\right)}\right)^{-\frac{\log z^\frac 12}{5\log\left(\log x\log\log x\right)}}\ll \left(\log z\right)^{-\frac{\log z}{11\log\log x}}\ll x^{-\frac{\alpha\kappa}{250}}.$$
We can handle the sum dyadically and apply (\ref{smoothbd}), finding
\begin{gather}\label{dyad}\sum_{\substack{z^\frac 12<c\leq z \\ p(c)\leq (\log x\log \log x)^5}}\frac{1}{c}\leq \sum_{j=0}^{\frac{\log z}{\log 2}}\frac{1}{2^jz^\frac 12}\sum_{\substack{2^jz^\frac 12<c\leq 2^{j+1}z^\frac 12 \\ p(c)\leq (\log x\log \log x)^5}}1 \ll \sum_{j=0}^{\frac{\log z}{\log 2}}\frac{1}{2^jz^\frac 12}\left[2^{j+1}z^\frac 12v\right]\ll v\log z\ll x^{-\frac{\alpha\kappa}{300}}.
\end{gather}
So
$$\sum_{n\in III}f(n)\ll \frac{Yx^{\epsilon-\frac{\alpha\kappa}{300}}}{k}\ll \frac{Yx^{-\epsilon'}}{k}$$
for some $\epsilon'>0$.

\subsection{Class IV}\label{IV}
This class gives us
$$\sum_{n\in IV}f(n)\leq \sum_{\substack{z^\frac 12<c\in \mathcal T[1,z] \\ (c,k)=1
}}f(c)\sum_{\substack{\frac{X-Y}c\leq d\leq \frac{X}{c} \\ d\equiv b\overline{c}\pmod k \\ q(d)>p(c) \\ (\log x\log\log x)^5< q(d)\leq z^\frac 12}}f(d).$$
Let 
$$r_0=\left[\frac{\log z}{5\log(\log x\log \log x)}\right],$$
where we now have $(\log x\log \log x)^5>z^{\frac{1}{r_0+1}}$.  Let $r$ be such that $2\leq r\leq r_0$, and consider those $n$ for which $z^{\frac{1}{r+1}}<q(d_n)\leq z^\frac 1r$. For such $n$, we know that the number of prime factors of $d_n$ is bounded by
$$\Omega(d_n)\leq \frac{\log x}{\log(q(d_n))}\leq \frac{20r}{\alpha\kappa}.$$
So
$$f(d_n)\leq \left(A_1(\log\log x)^\beta\right)^{\frac{20r}{\alpha\kappa}}\leq  A_1^{\frac{20r}{\alpha\kappa}}\left(\log\log x\right)^{\frac{20r}{41}}.$$
Hence
\begin{align*}
\sum_{n\in IV}f(n)\leq  &\sum_{2\leq r\leq r_0}A_1^{\frac{20r}{\alpha\kappa}}\left(\log\log x\right)^{\frac{20r}{41}}\sum_{\substack{z^\frac 12<c\leq z \\ p(c)<z^\frac 1r \\ c\in \mathcal T}}f(c)\Phi\left(\frac Xc,\frac Yc,z^\frac{1}{r+1};k,b\overline{c}\right)\\
\leq  &\sum_{2\leq r\leq r_0}A_1^{\frac{20r}{\alpha\kappa}}\left(\log\log x\right)^{\frac{20r}{41}}\left(\frac{Y(r+1)}{\phi(k)\log z}\right)\sum_{\substack{z^\frac 12<c\leq z \\ p(c)<z^\frac 1r \\ c\in \mathcal T}}\frac{f(c)}{c}
\end{align*}
by (\ref{rough}).

Let $$r_1=(\log\log x)^{1-\frac{\beta}{2}}.$$

We can split the sum over $r$ into
$$\sum_{2\leq r\leq r_1-1}+\sum_{r_1\leq r\leq r_0}.$$

For the former sum over $r$, we bound the sum over $c$ using (\ref{psum4simple}):
\begin{align*}
\sum_{2\leq r\leq r_1-1}&A_1^{\frac{20r}{\alpha\kappa}}\left(\log\log x\right)^{\frac{20r}{41}}\left(\frac{Y(r+1)}{\phi(k)\log z}\right)\sum_{\substack{z^\frac 12<c\leq z \\ p(c)<z^\frac 1r \\ c\in \mathcal T}}\frac{f(c)}{c}\\
\ll &\sum_{2\leq r\leq r_1-1}\left(\frac{Y(r+1)}{\phi(k)\log z}\right)A_1^{\frac{20r}{\alpha\kappa}}\left(\log\log x\right)^{\frac{20r}{41}}\exp\left(\sum_{\substack{J<p\leq z \\ p\nmid k}}\frac{f(p)}{p}\right)\\
\ll &\left(\frac{Yr_1}{\phi(k)\log z}\right)A_1^{\frac{20r_1}{\alpha\kappa}}\left(\log\log x\right)^{\frac{20r_1}{41}}\exp\left(\sum_{\substack{J<p\leq z \\ p\nmid k}}\frac{f(p)}{p}\right)\\
\ll &\left(\frac{Y}{\phi(k)\log z}\right)\exp\left((\log\log x)^{1-\frac{\beta}{3}}+\sum_{\substack{J<p\leq z \\ p\nmid k}}\frac{f(p)}{p}\right)
\end{align*}
Clearly, this is $\ll \frac{Y}{\phi(k)(\log x)^{1-\epsilon}}$ for any $\epsilon>0$.


Turning now to the interval where $r\geq r_1$, we can apply Lemma \ref{lemmalogr}.  Define
\begin{gather*}g(r)=\frac 12\left(1-\frac 54\beta \right)\log r,\\
h(r)=\frac{20\log A_1}{\alpha\kappa}+\frac{20}{41}\log\log\log x+\frac{2\log r}{r}+2A_1 (\log\log x)^{\beta}r^{-\frac 54\beta }
\end{gather*}
We can then bound the remaining sum as
\begin{align*}
\ll &\sum_{r_1\leq r\leq r_0}A_1^{\frac{20r}{\alpha\kappa}}\left(\log\log x\right)^{\frac{20r}{41}}\left(\frac{Y(r+1)}{\phi(k)\log z}\right)\\
&\cdot \exp\left(-\frac 12\left(1-\frac 54\beta\right)r\log r+\sum_{\substack{J<p\leq z^\frac 1r \\ p\nmid k}}\frac{f(p)}{p}+2A_1 (\log\log x)^{\beta}r^{1-\frac 54\beta}\right)\\
\ll &\sum_{r_1\leq r\leq r_0}\left(\frac{Y}{\phi(k)\log z}\right)\exp\left(-r(g(r)-h(r))+\sum_{\substack{J<p\leq X \\ p\nmid k}}\frac{f(p)}{p}\right).
\end{align*}
Now, since $\beta<\frac{\alpha\kappa}{41}\leq \frac{1}{164}$,
we see that
\begin{align*}g(r_1)=&\frac 12\left(1-\frac 74\beta+\frac 58\beta^2+o(1)\right)\log\log\log x\\
\geq &\left(\frac 12-\frac{1}{164}+o(1)\right)\log\log\log x.
\end{align*}
Meanwhile,
\begin{align*}h(r_1)=&\frac{20}{41}\log\log\log x+2A_1 (\log\log x)^{\beta}(\log \log x)^{-\frac 54\beta+\frac 58\beta^2}+O(1)\\
=&\left(\frac 12-\frac{1}{82}+o(1)\right)\log\log\log x.
\end{align*}
So 
$$g(r_1)-h(r_1)\geq \left(\frac{1}{60}+o(1)\right)\log\log\log x.$$
Noting also that $g'(r)>h'(r)$ for any $r\in [r_1,r_0]$, we can bound the above sum over $r$ as
\begin{align*}\ll &\sum_{r_1\leq r\leq r_0}\left(\frac{Y}{\phi(k)\log z}\right)\exp\left(-\frac{r}{60}\log\log\log x+\sum_{\substack{J<p\leq X \\ p\nmid k}}\frac{f(p)}{p}\right)\\
\ll &\left(\frac{Y}{\phi(k)\log z}\right)\exp\left(\sum_{\substack{J<p\leq X \\ p\nmid k}}\frac{f(p)}{p}\right)
\end{align*}
Putting these together,
\begin{align*}\sum_{n\in IV}f(n)\ll & \left(\frac{Y}{\phi(k)\log x}\right)\exp\left((\log\log x)^{1-\frac{\beta}{3}}+\sum_{\substack{J<p\leq X \\ p\nmid k}}\frac{f(p)}{p}\right)\\
\ll & \left(\frac{Y}{\phi(k)(\log x)^{1-\epsilon}}\right)\exp\left(\sum_{\substack{J<p\leq X \\ p\nmid k}}\frac{f(p)}{p}\right)
\end{align*}
for any $\epsilon>0$.  This completes the proof of Theorem \ref{roughbd} and hence of Theorem \ref{Maintheorem}.

\section{Alterations for Smooth-Supported Functions}\label{smooth}

Here, we alter the steps of the previous section to prove the analogous theorems for smooth-supported functions.  We note first that $\log Q/\log x=o(1)$, else the two bounds in Theorem \ref{MnThm2} are just Shiu's theorem and Theorem \ref{Maintheorem}, respectively.

We use the same definitions for the classes as in the previous section.  For a smooth-supported function, Class I cannot happen, since if $q(d_n)>z^\frac 12>Q$ then $n$ is not $Q$-smooth.  For Classes II and III, we can use the same bounds as in the non-smooth case.  So we require only class IV, as the first three classes give
$$\sum_{\substack{t\in \mathcal T[X-Y,X] \\ t\equiv a\pmod k}}f(t)=\sum_{n\in IV}f(n)+O\left(\frac{Y}{kz^{\frac 14-\epsilon}}\right)$$

We consider first the case when $f\in M_Q$.  Since $f(p)$ is now bounded, we could apply Lemma 4 of Shiu's original paper \cite{Sh}, which states that if $f\in M$ then
\begin{gather}\label{Shlem}\sum_{\substack{n\geq z^\frac 12 \\ p(n)\leq z^\frac 1r \\ (n,k)=1 \\ n\in \mathcal T}}\frac{f(n)}{n}\ll \exp\left(-\frac{1}{10}r\log r+\sum_{\substack{p\leq z \\ p\nmid k}}\frac{f(p)}{p}\right).
\end{gather}
The proof there is very similar to our proof of Lemma \ref{lemmalogr} except that Shiu takes $\delta=1-\frac{r\log r}{4\log z}$, whereas we take $\delta=1-\frac{(1-\frac 54\beta )r\log r}{\log z}$.  If we were to apply our $\delta$ in Shiu's original Lemma 4, we would instead have
\begin{gather}\label{Shlem4}\sum_{\substack{n\geq z^\frac 12 \\ p(n)\leq z^\frac 1r \\ (n,k)=1 \\ n\in \mathcal T}}\frac{f(n)}{n}\ll \exp\left(-\frac 12\left(1-\frac 54\beta\right)r\log r+\sum_{\substack{p\leq z \\ p\nmid k}}\frac{f(p)}{p}\right),
\end{gather}
where the $\ll$ depends on our choice of $\beta$.

Now, as we are dealing with Class IV, we note that since $z^\frac 12<c\leq z$, we must have $d>x^{1-\frac{\alpha \kappa}{20}}$.  So if $n$ is such that $f(n)\neq 0$ then $n$ must have at least $\frac{\left(1-\frac{\alpha \kappa}{20}\right)\log x}{\log Q}$ prime factors.  Recalling that
$$\Omega(d_n)\leq \frac{20r}{\alpha \kappa},$$
we must then have
$$\frac{\alpha\kappa \left(1-\frac{\alpha \kappa}{20}\right)\log x}{20\log Q}\leq r$$
or else $n$ is not $Q$-smooth.  Let 
$$r_2=\frac{\alpha\kappa \left(1-\frac{\alpha \kappa}{20}\right)\log x}{20\log Q}=\frac{\alpha\kappa \left(1-\frac{\alpha \kappa}{20}\right)}{20}u.$$
Then 
\begin{align*}
\sum_{n\in IV}f(n)\leq  &\sum_{r_2\leq r\leq r_0}A_1^{\frac{20r}{\alpha\kappa}}\left(\frac{Y(r+1)}{\phi(k)\log z}\right)\sum_{\substack{z^\frac 12<c\leq z \\ p(c)<z^\frac 1r }}\frac{f(c)}{c}\\
\ll  &\sum_{r_2\leq r\leq r_0}\left(\frac{Y}{\phi(k)\log z}\right)\exp\left(-\frac 12\left(1-\frac 54\beta\right)r\log r+\sum_{\substack{p\leq Q \\ p\nmid k}}\frac{f(p)}{p}+\frac{20r}{\alpha\kappa}\log A_1+2\log r\right)\\
\ll  &\sum_{r_2\leq r\leq r_0}\left(\frac{Y}{\phi(k)\log z}\right)\exp\left(-\frac 12\left(1-\frac 32\beta\right)r\log r\right)\exp\left(\sum_{\substack{p\leq Q \\ p\nmid k}}\frac{f(p)}{p}\right)\\
\ll  &\left(\frac{Y}{\phi(k)\log z}\right)\exp\left(-\left(\frac 12-\beta\right)r_2\log r_2\right)\exp\left(\sum_{\substack{p\leq Q \\ p\nmid k}}\frac{f(p)}{p}\right).
\end{align*}
Noting that 
$$\exp\left(-\left(\frac 12-\beta\right)r_2\log r_2\right)=u^{-\frac{\alpha\kappa}{40}\left(1-2\beta\right)u},$$
we can let $$C_1=\frac{\alpha\kappa}{41}.$$ 
This then gives us 
$$\sum_{n\in IV}f(n)\ll \rho(u)^{C_1+o(1)}\left(\frac{Y}{\phi(k)\log x}\right)\left(\sum_{\substack{J<p\leq Q \\ p\nmid k}}\frac{f(p)}{p}\right),$$
which proves the first half of the theorem.

For the case of $f\in \mathcal M_Q(\beta)$, we proceed similarly, except that we use Lemma \ref{lemmalogr}.  So again,

\begin{align*}
\sum_{n\in IV}f(n)
&\leq  \sum_{r_2\leq r\leq r_0}A_1^{\frac{20r}{\alpha\kappa}}\left(\log\log x\right)^{\frac{20r}{41}}\left(\frac{Y(r+1)}{\phi(k)\log z}\right)\sum_{\substack{z^\frac 12<c\leq z \\ p(c)<z^\frac 1r \\ c\in \mathcal T}}\frac{f(c)}{c}\\
&\ll  \sum_{r_2\leq r\leq r_0}\left(\frac{Y}{\phi(k)\log z}\right)\exp\left(-r(g(r)-h(r))+\sum_{\substack{J<p\leq Q \\ p\nmid k}}\frac{f(p)}{p}\right),
\end{align*}
where $g$ and $h$ are as in the previous section.

Since $r_2>r_1$, we note as before that $g'(r)>h'(r)$ for all $r\in [r_2,r_0]$.  Here, however, we also have
\begin{gather*}g(r_2)=\frac 12\left(1-\frac 54\beta \right)\left[\log u+\log\left(\frac{\alpha\kappa \left(1-\frac{\alpha \kappa}{20}\right)}{20}\right)\right].
\end{gather*}
Moreover, since $\frac{\log x}{\log Q}\geq \log \log x$ by the assumption in Theorem \ref{MnThm2}, we know that
\begin{align*}
h(r_2)=&\frac{20}{41}\log\log\log x+O(1)+2A_1 (\log\log x)^{\beta}\left(\frac{\alpha\kappa \left(1-\frac{\alpha \kappa}{20}\right)}{20}u\right)^{-\frac 54\beta }\\
=&\left(\frac{20}{41}+o(1)\right)\log\log\log x\leq \left(\frac{20}{41}+o(1)\right)\log u.\end{align*}
Since $\frac{5}{8}\beta<\frac{5}{1312}$, we have
$$h(r_2)-g(r_2)\leq -\left(\frac{11}{1312}+o(1)\right)\log \left(u\right).$$
So letting 
$$C_2=\frac{5}{656}\left(\frac{\alpha\kappa \left(1-\frac{\alpha \kappa}{20}\right)}{20}\right),$$
we have
$$\exp\left(-r_2(g(r_2)-h(r_2))\right)\ll \rho(u)^{C_2+o(1)}.$$
Thus,
$$\sum_{n\in IV}f(n)\ll \rho(u)^{C_2+o(1)}\left(\frac{Y}{\phi(k)\log x}\right)\left(\sum_{\substack{J<p\leq Q \\ p\nmid k}}\frac{f(p)}{p}\right).$$
This completes the proof of Theorem \ref{roughbd2}, and by extension, the proof of Theorem \ref{MnThm2}.

\section{Application: the $d$-fold Divisor Function}
As in Section 6 of Shiu's original paper, we use this result to prove a statement about the divisor function.  Here, we let $\tau_d$ denote the $d$-fold divisor function, and let $\tau=\tau_2$.
\begin{varthmdfold}
Let $x$, $y$, and $k$ be as in Theorem \ref{Maintheorem}.  Let $d\geq 2$, and let $R=\frac{\log\log\log x}{g(x)}$ for some $g(x)$ such that $g(x)\to\infty$ as $x\to\infty$.  Then for any $\epsilon>0$,
$$\sum_{\substack{x-y\leq n\leq x \\ n\equiv a\pmod k}}\tau_d(n)^{R}\ll \frac{y(\log x)^\epsilon}{\phi(k)}\left(\frac{\phi(k)}{k}\log x\right)^{d^{R}-1}.$$
Moreover, let $Q\leq e^{\frac{\log x\log\log\log x}{\log\log x}}$.  Then
$$\sum_{\substack{x-y\leq n\leq x \\ n\equiv a\pmod k \\ n\in \mathcal S(Q)}}\tau_d(n)^{R}\ll \rho(u)^{C_2+o(1)}\frac{y(\log x)^\epsilon}{\phi(k)}\left(\frac{\phi(k)}{k}\log x\right)^{d^{R}-1}.$$
\end{varthmdfold}
\begin{proof}
At the end of Section 4, we showed that Conditions 1$\beta$ and 2 hold for the binary divisor function $\tau$.  We prove now that Conditions 1$\beta$ and 2 also hold for the $d$-fold divisor function $\tau_d$.

For any prime $p$,
$$\tau_d(p)^{R}=d^{\frac{\log\log\log x}{g(x)}}=(\log\log x)^{\frac{\log d}{g(x)}}.$$
So Condition 1$\beta$ holds for $\beta$ as given above.

Moreover, for any $n\geq 3$, we have an effective upper bound for the binary divisor function \cite{NR}:
$$\tau_2(n)^{R}\leq n^{\frac{1.5379 (\log\log\log x)}{g(x)\log\log n}}$$
Noting that 
$$\tau_{d+1}(n)=\sum_{m|n}\tau_d\left(\frac nm\right)\leq \sum_{m|n}\tau_d(n)=\tau_2(n)\tau_d(n),$$
we can inductively bound $\tau_d$ by
$$\tau_d(n)\leq \tau_2(n)^{d-1}.$$
Hence
$$\tau_d(n)^{R}\leq n^{\frac{1.5379 (d-1)(\log\log\log x)}{g(x)\log\log n}}$$
The right-hand side is obviously an increasing function on $n$.  So if $n\leq (\log x)^{10}$ then this is 
$$\leq (\log x)^{\frac{15.379 (d-1)}{g(x)}},$$
which is clearly $\ll (\log x)^\epsilon$ for any $\epsilon$.  If $n>(\log x)^{10}$ then this is
$$\leq n^{\frac{1.5379  (d-1)}{g(x)}},$$
which is $\ll n^\epsilon$ for every $\epsilon>0$.  So Condition 2 also holds.  Hence, $\tau_d(n)^R\in M(\beta)$, where $\beta$ is as in Theorem \ref{Maintheorem}.

Applying Theorem \ref{Maintheorem}, we then have
$$\sum_{\substack{x-y\leq n\leq x \\ n\equiv a\pmod k}}\tau_d(n)^{R}\ll \frac{y}{\phi(k)(\log x)^{1-\epsilon}}\exp\left(d^R\sum_{\substack{p\leq x \\ p\nmid k}}\frac{1}{p}\right)\ll \frac{y(\log x)^\epsilon}{k} \left(\frac{\phi(k)}{k}\log x\right)^{d^{R}-1}$$
as in the original proof in \cite{Sh}.

The proof of the second half of the theorem is nearly identical except that we apply Theorem \ref{MnThm2} instead of Theorem \ref{Maintheorem}.

\end{proof}

Note that if we were to prove the smooth portion of Theorem \ref{dfold} by restricting our sum over the primes in the exponential term to $\sum_{p\leq Q}\frac{\tau_d(p)}p$, we would only save a term that is $\rho(u)^{o(1)}$, and hence we do not pursue this idea further here.

\section{Application: Smooth Numbers in Short Intervals}
Finally, we turn to the results of \cite{So} and \cite{Go}.  We find the following improvement.

\begin{vtsmooththm}
Let $Q$, $x$, and $y$ satisfy the hypotheses in either Theorem \ref{SoundThm} or \ref{GoThm}, where in the former case we also assume the Riemann Hypothesis.  Additionally, assume that $u\geq (\log\log x)^2.$  Then there exists a constant $C_3$ such that 
\begin{gather}
\Psi(x+y,Q)-\Psi(x,Q)\gg y\rho(u)^{1+\frac{C_3\log\log \log\log x}{\log\log \log x}}.
\end{gather}
In particular, we can take $C_3>1-C_2$, where $C_2$ is as defined in Theorem \ref{MnThm2}.

\end{vtsmooththm}


\begin{proof}
In both of these cases, the first step is to find a lower bound for $$\sum_{\substack{x\leq n\leq x+y \\ n\in \mathcal S(Q) }}\tau_d(n)$$
for some $d$ and for $y$ as in the theorem.

We begin with \cite{So}.  Using the notation of that paper, we define $\delta$ such that
$$xe^{2\delta}=x+y.$$
Naturally, this means that $$\delta=\frac{y}{2x}(1+o(1)).$$
The paper also defines
$$I=\sum_{x\leq n\leq x+y}\sum_{\substack{n=rm_1m_2 \\ m_1,m_2\in \mathcal S(Q) \\ \sqrt xQ^{-\frac 13}\leq m_1,m_2\leq xQ^{-\frac 14} }}\Lambda(r)\min\{\log \frac{e^{2\delta}x}{n},\log \frac nx\}.$$
Note that $r\leq Q$, and hence $r$ is also $Q$-smooth.  So
$$I\ll \delta\log Q\sum_{\substack{x\leq n\leq x+y \\ n\in \mathcal S(Q) }}\tau_3(n).$$
This paper eventually finds that 
$$I\gg x\delta^2M(1)^2/2,$$
where $M(1)\geq \rho(u/2)\log Q/24$.
Putting these together, we have
$$x\delta\rho(u/2)^2\log Q\ll \log Q\sum_{\substack{x\leq n\leq x+y \\ n\in \mathcal S(Q) }}\tau_3(n),$$
which can be re-expressed as
\begin{gather}\label{zrho}y\rho(u/2)^2\ll \sum_{\substack{x\leq n\leq x+y \\ n\in \mathcal S(Q) }}\tau_3(n).
\end{gather}
Moving on to the unconditional case of \cite{Go}, the process is similar.  Near the end of Section 6 of \cite{Go}, the author of that paper finds that 
$$\rho(u/2)^2h\sqrt x\ll \sum_{\substack{x-2h\sqrt x<n_1n_2\leq x+2h\sqrt x \\ n_1\sim \sqrt x \\ n_1,n_2\in \mathcal S(Q)}}1.$$
Clearly,
$$\sum_{\substack{x-2h\sqrt x<n_1n_2\leq x+2h\sqrt x \\ n_1\sim \sqrt x \\ n_1,n_2\in \mathcal S(Q)}}1\ll \sum_{\substack{x-2h\sqrt x<n\leq x+2h\sqrt x \\ n\in \mathcal S(Q)}}\tau(n),$$
and hence if we rewrite $x-2h\sqrt x$ as $x_0-y$ and let $y=4h\sqrt x$ then we have
\begin{gather}\label{zrho2}y\rho(u/2)^2\ll \sum_{\substack{x_0-y<n\leq x_0 \\ n\in \mathcal S(Q)}}\tau(n).\end{gather}
Now, let 
$$R=\frac{\log\log \log x}{\log \log \log \log x}.$$
Then by Holder's inequality,
$$\sum_{\substack{x_0-y<n\leq x_0 \\ n\in \mathcal S(Q)}}\tau_d(n)\leq \left(\sum_{\substack{x_0-y<n\leq x_0 \\ n\in \mathcal S(Q)}}1^{\frac{1}{1-\frac 1R}}\right)^{1-\frac 1R}\left(\sum_{\substack{x_0-y<n\leq x_0 \\ n\in \mathcal S(Q)}}\tau_d(n)^{R}\right)^\frac 1R.$$
Applying Theorem \ref{dfold}, we that the above is
\begin{align*}\ll &\left(\sum_{\substack{x_0-y<n\leq x_0 \\ n\in \mathcal S(Q)}}1\right)^{1-\frac 1R}\left(y \rho(u)^{C_2+o(1)}\left(\log x\right)^{d^{R}-1+\epsilon}\right)^\frac 1R\\
\ll &\left(\sum_{\substack{x_0-y<n\leq x_0 \\ n\in \mathcal S(Q)}}1\right)^{1-\frac 1R}y^\frac 1R \rho(u)^{\frac{\log\log\log\log x}{\log\log\log x}\left(C_2+o(1)\right)}\left(\log x\right)^{\left(\log\log x\right)^{\frac{\log d}{\log\log\log\log x}}}\\
\ll & \left(\sum_{\substack{x_0-y<n\leq x_0 \\ n\in \mathcal S(Q)}}1\right)^{1-\frac 1R}y^\frac 1R \rho(u)^{\frac{\log\log\log\log x}{\log\log\log x}\left(C_2+o(1)\right)}e^{u^\frac 23}.
\end{align*}
since $u\geq (\log\log x)^2$. 

Recall that
$$\rho(u)=\exp\left(-u\left(\log u+\log\log u-1+\frac{\log\log u-1}{\log u}+O\left(\left(\frac{\log\log u}{\log u}\right)^2\right)\right)\right).$$
So
$$\rho(u/2)^2=\exp\left(-u\left(\log u-\log 2+\log\log u-1+O\left(\frac{\log\log u}{\log u}\right)\right)\right).$$
Plugging into (\ref{zrho}) or (\ref{zrho2}), we then have
\begin{align*}
y\exp\left(-u\left(\log u+\log\log u+O(1)\right)\right)\ll \left(\sum_{\substack{x_0-y<n\leq x_0 \\ n\in \mathcal S(Q)}}1\right)^{1-\frac 1R}y^\frac 1R\rho(u)^{\frac{\log\log\log\log x}{\log\log\log x}\left(C_2+o(1)\right)}e^{u^\frac 23},
\end{align*}
or
\begin{align*}
y\left[\exp\left(-\left(1-\frac{C_2}{R}\right)u\left(\log u+\log\log u+O(1)\right)\right)\right]^{1+\frac{1}{R}}\ll \sum_{\substack{x_0-y<n\leq x_0 \\ n\in \mathcal S(Q)}}1,
\end{align*}
Distributing the exponent gives us a left-hand side of
$$y\exp\left(-\left(1+\frac{1-C_2}{R}-\frac{C_2}{R^2}\right)u\left(\log u+\log\log u+O(1)\right)\right).$$
Plugging back in for $R$ gives
\begin{align*}
y\rho(u)^{1+\frac{(1-C_2)\log\log\log \log x}{\log\log\log x}+O\left(\left(\frac{\log\log \log u}{\log \log \log \log u}\right)^2\right)}\ll \sum_{\substack{x_0-y<n\leq x_0 \\ n\in \mathcal S(Q)}}1,
\end{align*}
which is as required.
\end{proof}
\section{Acknowledgements}

I would like to thank Sarvagya Jain for helpful comments and suggestions.  I would also like to thank the referee for a very thorough reading of this paper and many helpful suggestions.

\bibliographystyle{line}

\end{document}